\documentclass[11pt]{article}
\usepackage{amsthm}

\usepackage{graphicx}

\usepackage[english]{babel}
\usepackage{amssymb}
\usepackage{amsmath}
\usepackage{enumerate}

\usepackage{epsfig}
\usepackage{subfigure}

\newcommand{\old}[1]{{}}

\newcommand{\bb}{\mathbb}
\newcommand{\conv}{\mathbf{conv}}

\newcommand{\R}{\bb R}
\newcommand{\T}{\bb T}

\newcommand{\Z}{\bb Z}

\newcommand{\intr}{\mathbf{int}}

\newcommand{\vol}{\mathbf{vol}}

\newtheorem{prop}{Proposition}
\newtheorem{theorem}[prop]{Theorem}
\newtheorem{lemma}[prop]{Lemma}

\newtheorem{claim}{Claim}
\newtheorem{cor}[prop]{Corollary}

\newtheorem{question}[prop]{Question}

\def\st{\mid}

\begin{document}
\title{Unique Minimal Liftings for Simplicial Polytopes}

\author{Amitabh Basu\thanks{Department of Mathematics, University of California, Davis,
abasu@math.ucdavis.edu}\and
G\'erard Cornu\'ejols\thanks{Tepper School of Business, Carnegie Mellon University, Pittsburgh, gc0v@andrew.cmu.edu \newline Supported by NSF grant CMMI1024554 and ONR grant N00014-09-1-0033.}\and
Matthias K\"oppe\thanks{Department of Mathematics, University of California, Davis,
mkoeppe@math.ucdavis.edu}}
\date{March 2011, revised September 2011}
\maketitle
\begin{abstract}
 For a minimal inequality derived from a maximal lattice-free simplicial polytope in $\R^n$, we investigate the region where minimal liftings are uniquely defined, and we characterize when this region covers $\R^n$. We then use this characterization to show that a minimal inequality derived from a maximal lattice-free simplex in $\R^n$ with exactly one lattice point in the relative interior of each facet has a unique minimal lifting if and only if all the vertices of the simplex are lattice points.
\end{abstract}

\section{Introduction}

An important topic in integer programming is the generation of valid inequalities for the
Gomory-Johnson infinite relaxation, defined for $f \in \mathbb{R}^n \setminus \mathbb{Z}^n$:

\begin{equation} \label{mod:infiniteZ} \begin{array}{l}
\displaystyle{f + \sum_{r\in \R^n}r s_r} \;\; \in \mathbb{Z}^n \\
s_r\in \Z_+,\: \forall r\in \R^n,\\
s \text{ has finite support.}
\end{array}
\end{equation}

A function $\pi\colon \R^n \rightarrow \R_+$ is said to be {\em valid}
for \eqref{mod:infiniteZ} if the inequality $\sum_{r\in\R^n}\pi(r)s_r
\geq 1$ holds for every $s$ satisfying \eqref{mod:infiniteZ}. To generate valid
inequalities for \eqref{mod:infiniteZ}, one possible approach is to lift valid
inequalities for a simpler model where the variables $s_r$ are continuous:

\begin{equation} \label{mod:infiniteR} \begin{array}{l}
\displaystyle{f + \sum_{r\in \R^n}r s_r} \;\; \in \mathbb{Z}^n \\
s_r\in \R_+,\: \forall r\in \R^n,\\
s \text{ has finite support.}
\end{array}
\end{equation}

A function $\psi\colon \R^n \rightarrow \R_+$ is said to be {\em valid}
for \eqref{mod:infiniteR} if the inequality $\sum_{r\in\R^n}\psi(r)s_r
\geq 1$ holds for every $s$ satisfying \eqref{mod:infiniteR}.
Given a valid function $\psi$ for \eqref{mod:infiniteR}, a function $\pi$ is a {\em
lifting} of $\psi$ if $\pi \leq \psi$ and $\pi$ is valid for (\ref{mod:infiniteZ}).
One is interested only in (pointwise) {\em minimal} valid functions, since the non-minimal
valid functions are implied by the minimal ones.
Minimal valid functions $\psi$ for \eqref{mod:infiniteR} have a simple characterization
in terms of maximal lattice-free convex sets~\cite{BorCor}. By a {\em lattice-free convex set}, we mean a convex set that does not contain any integer point in its interior.
Specifically, \cite{BorCor, BaCoCoZaM09} show that if $B \subseteq \R^n$ is a maximal lattice-free convex set containing
$f$ in its interior, then the gauge of $B - f$ is a minimal valid
function for~\eqref{mod:infiniteR}, and every minimal valid function arises this way.
Lov\'asz~\cite{lo} showed that maximal lattice-free convex sets are polyhedra.
Thus $B$ can be uniquely written in the form
\begin{eqnarray}\label{eq:B}
B=\{\,x\in \R^n \st a^i \cdot (x-f)\le 1,\;i\in I\,\},
\end{eqnarray} where $I$ is a finite set of indices and $a^i \cdot (x-f)\leq 1$ is facet-defining for $B$ for every $i\in I$.
The gauge function of $B - f$ can then be written in the form
\begin{equation}\label{eq:psiB}
\psi(r)= \max_{i\in I} a^i \cdot r, \quad \forall r\in \mathbb{R}^n.
\end{equation}
Given a minimal valid function $\psi$ and the corresponding maximal lattice-free convex set $B$, Basu, Campelo, Conforti, Cornu\'ejols and Zambelli \cite{lifting} characterized the region $\bar R \subset \mathbb{R}^n$
where $\pi = \psi$ for every lifting $\pi$ of $\psi$. This region will henceforth be referred to as the {\em lifting region}. The case where  $\bar R + \mathbb{Z}^n$ covers $\mathbb{R}^n$ is particularly interesting since, in this case, there is a unique minimal lifting $\pi$, which can be computed using the so-called ``trivial fill-in'' procedure \cite{dw}. In fact, \cite{bcccz2} shows that $\psi$ has a unique lifting $\pi$ if and only if $\bar R+\Z^n$ covers $\R^n$.

\subsubsection*{Our Results}

Let $B$ be a maximal lattice-free convex body.
  We say that $B$ is a {\em body with a unique lifting for all $f \in \intr(B)$} if the minimal valid function $\psi$ defined by \eqref{eq:B}--\eqref{eq:psiB} has a unique minimal lifting for {\em all} $f \in \intr(B)$. We say that $B$ is a {\em body with multiple liftings for all $f \in \intr(B)$} if $\psi$ has more than one minimal lifting for {\em all} $f \in \intr(B)$. In general, $\psi$ might have a unique minimal lifting for {\em some} $f \in \intr(B)$, and multiple liftings for other $f \in \intr(B)$. We prove that this never happens when $B$ is a simplicial polytope.

\begin{theorem}\label{body-with}
Let $B$ be a maximal lattice-free simplicial polytope in $\R^n$ ($n \geq 2$). Then $B$ is either a body with a unique lifting for all $f \in \intr(B)$, or a body with multiple liftings for all $f \in \intr(B)$.
\end{theorem}

Theorem~\ref{body-with} raises an interesting question. The family of simplicial maximal lattice-free polytopes is such that it partitions into bodies with unique liftings for all $f$ and bodies with multiple liftings for all $f$. It is not known if this happens for general maximal lattice-free bodies.\old{ Let $\mathcal{M}$ be the family of all maximal lattice-free convex bodies. Let $\mathcal{U} \subseteq \mathcal{M}$ be the subset of bodies with unique liftings for all $f$ in the interior and $\mathcal{N} \subseteq \mathcal{M}$ be the subset of bodies with multiple liftings for all $f$ in the interior. The following question is open:

\begin{question}
Is $\mathcal{M} = \mathcal{U} \cup \mathcal{N}$?
\end{question}}
In other words, the following question is open: does there exist a maximal lattice-free convex body $B$ such that it has a unique lifting for some $f \in \intr(B)$ and has more than one lifting for some other position of $f \in \intr(B)$? Or does the uniqueness of the lifting function depend only on the geometry of the maximal lattice-free body $B$ and not on the position of $f$?

In $\mathbb{R}^2$, Dey and Wolsey \cite{dw} showed that a lattice-free triangle $B$ with exactly one integer point in the interior of each side is a body with a unique lifting for all $f \in \intr(B)$ if and only if the vertices of the triangle are themselves integer points. We generalize this result to $\R^n$.

\begin{theorem} \label{thm:main}
Let $\Delta$ be a simplex in $\R^n$ ($n \geq 2$) such that it is a maximal lattice-free convex body and each facet of $\Delta$ has exactly one integer point in its relative interior.  Then $\Delta$ is a body with a unique lifting for all $f \in \intr(\Delta)$ if and only if $\Delta$ is an affine unimodular transformation of $\conv\{0, ne^1, \ldots, ne^n\}$.
\end{theorem}

Note that combined with Theorem~\ref{body-with}, this theorem says that if $\Delta$ is {\em not} an affine unimodular transformation of $\conv\{0, ne^1, \ldots, ne^n\}$, then it is a body with multiple liftings for every $f \in \intr(\Delta)$.

Dey and Wolsey \cite{dw} also showed that a lattice-free triangle $B \subset
\mathbb{R}^2$ with a side containing more than one integer point in its
interior is a body with a unique lifting for all $f \in \intr(B)$.  The story
for simplices with multiple integer points in the relative interior of the
facets is substantially more complicated in higher dimensions. In fact, there
are examples of simplices in $\R^n$ for every $n\geq 3$ with multiple lattice
points in the interior of a facet and whose lifting region does not cover
$\R^n$ by lattice translations, implying that it is a body with multiple
liftings for all $f$ in the interior.

We prove the following interesting
theorem for a special class of simplices, which generalizes the results in
$\R^2$.  At the same time, it gives an infinite family of simplices for every $n\geq 3$ with multiple integer points in the relative interior of a facet whose lifting regions do not cover $\R^n$ by lattice translations. Explicit constructions for such examples are shown in Section~\ref{sec:construction}. We define a maximal lattice-free convex body $B$ to be {\em 2-partitionable} with hyperplanes $H_1, H_2$ if $H_1, H_2$ are two adjacent lattice hyperplanes and all the lattice points on the boundary of $B$ lie on $H_1 \cup H_2$.  This notion was utilized by Basu, Cornu\'ejols and Margot to characterize split rank for minimal inequalities in~\cite{bcm} and turns out to be an important notion.

\begin{theorem} \label{thm:multiple}
Let $\Delta \subset \R^{n+1}$ be a maximal lattice-free 2-partitionable simplex with hyperplanes $H_1, H_2$ such that $H_1$ defines a facet of $\Delta$ and this is the only facet of $\Delta$ with more than one lattice point in its relative interior. Then $\Delta$ is a body with a unique lifting for all $f \in \intr(\Delta)$ if and only if $\Delta \cap H_2$ is an affine unimodular transformation of $\conv\{0, ne^1, \ldots, ne^n\}$.
\end{theorem}

Again note that combined with Theorem~\ref{body-with}, this theorem says that if $\Delta \cap H_2$ is {\em not} an affine unimodular transformation of $\conv\{0, ne^1, \ldots, ne^n\}$, then $\Delta$ is a body with multiple liftings for every $f \in \intr(\Delta)$.

To prove these results,  we will need to investigate when the region $\bar R$ mentioned earlier has the property that $\bar R+\Z^n$ covers $\R^n$. For this purpose, we consider the torus $\T^n = \R^n/\Z^n$, equipped with the natural Lebesgue measure that assigns volume~1 to~$\R^n/\Z^n$.  It is clear that a
compact set $X\subseteq\R^n$ covers $\R^n$ by lattice translations, i.e.,
$X+\Z^n = \R^n$, if and only if $\vol_{\T^n} (X/\Z^n) = 1$, where $\vol_{\T^n}
(X/\Z^n)$ denotes the volume of $X/\Z^n$ in $\T^n$. (This is not necessarily
true for sets $X$ that are merely closed, as in this case $X/\Z^n$ is not necessarily closed.)

Let $B \subset \R^n$ be a maximal lattice-free simplicial polytope. We will investigate how the lifting region $\bar R$ varies as a function of $f$, for points $f \in \intr(B)$. Therefore it will be convenient to denote the lifting region by $\bar R(f)$ in the remainder.
One of the key ingredients in proving the above theorems is to show that $\vol_{\T^n}(\bar R(f)/\Z^n)$ is an affine function of the coordinates of $f$. This reduces the problem of checking the covering properties of $\bar R(f)$ when $f$ is at the vertices of $B$.

\begin{theorem} \label{lem:mod_vol}
Let $B \subset \R^n$ be a maximal lattice-free simplicial polytope and let $f \in \intr(B)$. Then $\vol_{\T^n}(\bar R(f)/\Z^n)$ is an affine function of the coordinates of $f$.
\end{theorem}

The proof of this theorem is the main task of Section~\ref{sec:simplicial}.
Theorem~\ref{body-with} will follow as a corollary.
Theorem~\ref{thm:main} will be proved in Section~\ref{sec:one-latt} and Theorem~\ref{thm:multiple} in Section~\ref{sec:multi}.

\section{Lifting Regions for Simplicial Polytopes}\label{sec:simplicial}

Let $B \subset \R^n$ be a maximal lattice-free simplicial polytope and let $f \in \intr(B)$.
We express $B$  as in \eqref{eq:B}.
For $i \in I$, we denote by $y^{ik}$ the integer points lying on the facet of $B$ defined by $a^i(x-f) \leq 1$, where $k \in K_i$ for some index set $K_i$. Since $B$ is simplicial, each facet has $n$ vertices; we denote these vertices by $v^{ij}$, $j = 1,\ldots, n$, $i\in I$. For $i\in I$, let $C_i$ be the simplicial cone generated by the rays $v^{ij} - f$, $j = 1, \ldots, n$. Let $R_{ik} = (f + C_i) \cap (y^{ik} - C_i)$. Define

\begin{equation}\label{eq:R} R(f) = \bigcup_{i \in I} \bigcup_{k \in K_i} R_{ik}.
\end{equation}

It was shown in \cite{lifting} that the lifting region $\bar R(f) = R(f) - f$. The lifting region $\bar R(f)$ can thus be obtained as the union of the parallelotopes $R_{ik} - f$, $i \in I$, $k \in K_i$ for simplicial polytopes. Since $\bar R(f)$ is a translation of $R(f)$, $\R^n$ is covered by lattice translates of $\bar R(f)$ if and only if $\R^n$ is covered by lattice translates of $R(f)$. Therefore, we will investigate the covering properties of $R(f)$ as defined in \eqref{eq:R} in the rest of the paper.

\begin{figure}
\centering \subfigure[A maximal lattice-free triangle with three integer points] {\label{fig:type_3}
\includegraphics[height=2.5in]{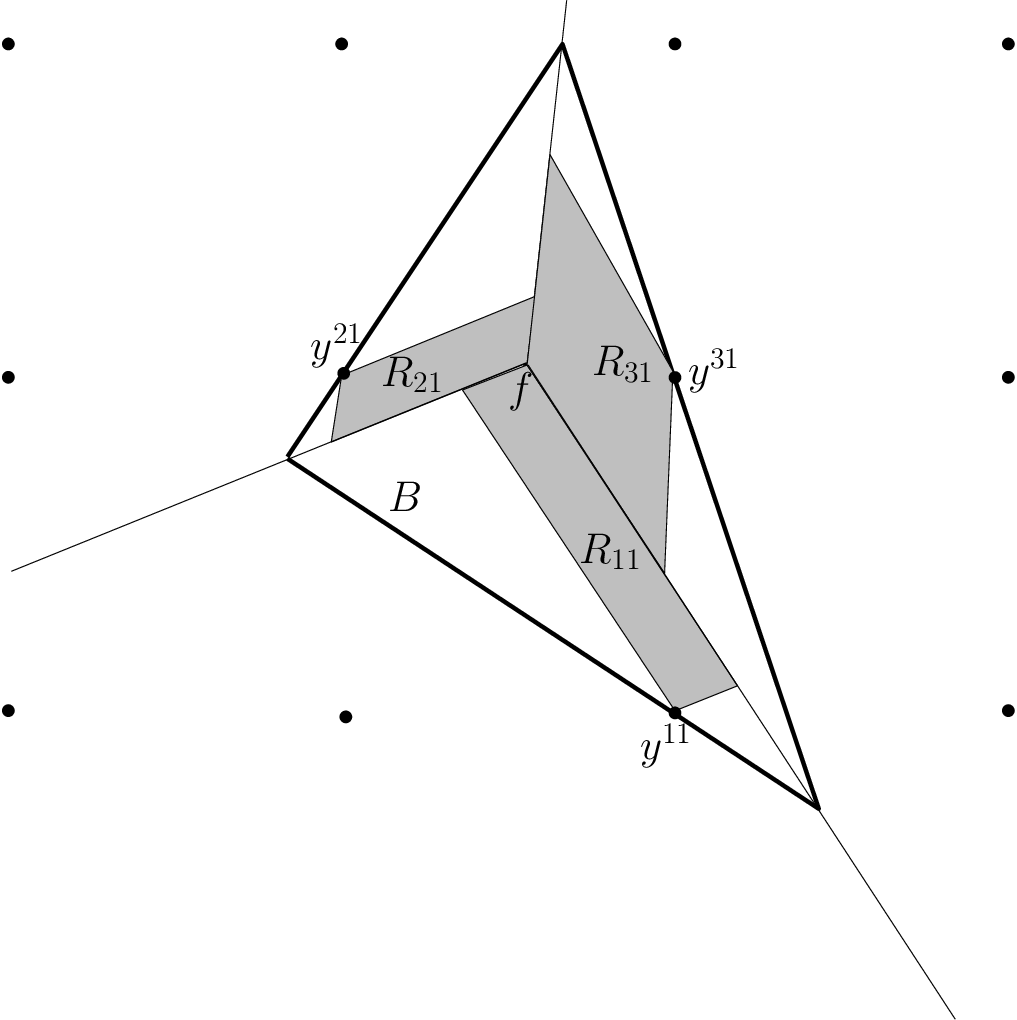}
}
\hspace{0.5in}\subfigure[A maximal lattice-free triangle with integer
vertices; this is an example where the region $R_{ik}$ is full-dimensional for $y^{1,2}, y^{2,2}, y^{3,2}$, but is
not full-dimensional for the integral vertices.] {\label{fig:type_1}
\includegraphics[height=2.0in]{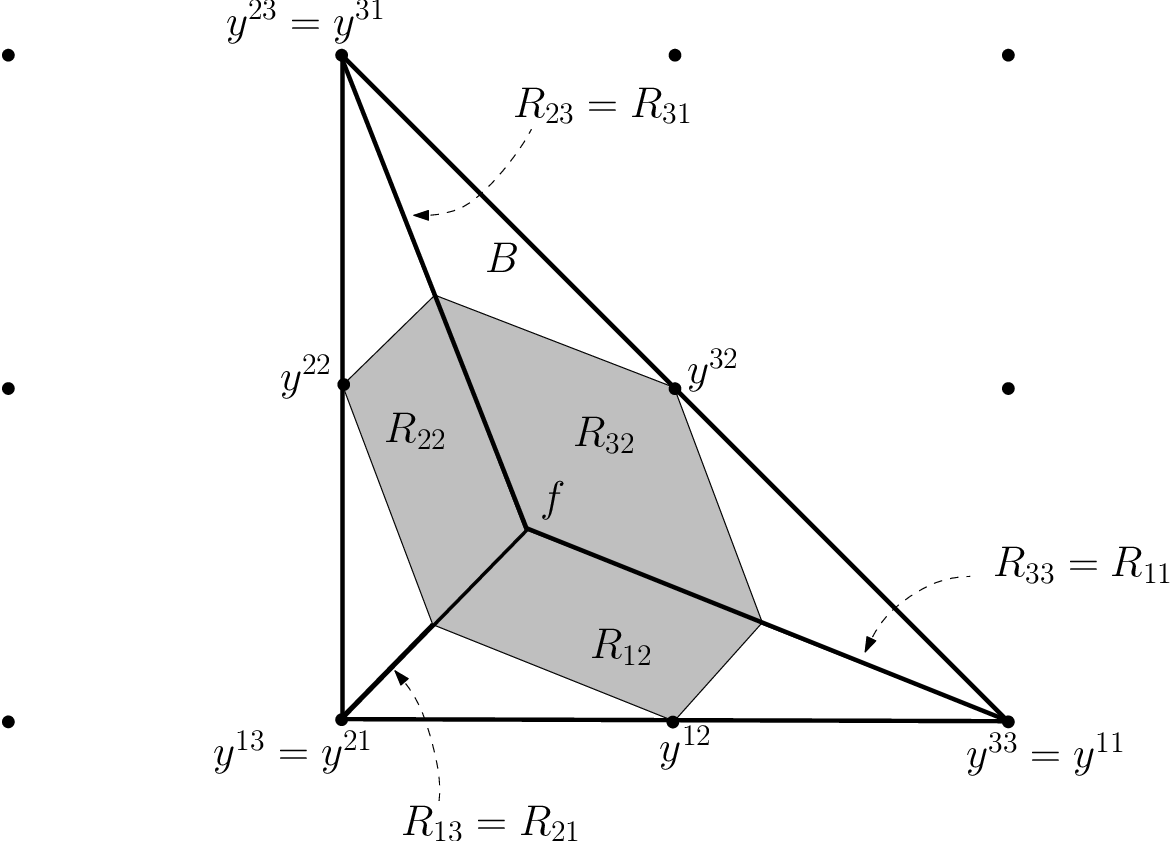}
}
\caption{Regions $R_{ik}$ for some
maximal lattice-free convex sets in the plane. The thick
dark line indicates the boundary of
$B$. For a particular $y^{ik}$, the dark gray regions denote~$R_{ik}$.}\label{fig:regions}
\end{figure}

Since $B$ is a simplicial polytope, each $y^{ik}$, $i \in I$, $k\in K_i$, can be uniquely expressed as
$$y^{ik} = \sum_{j=1}^n \lambda^{ik}_j v^{ij}$$
\noindent
where $\lambda^{ik}_j \geq 0$ and $\sum_{j=1}^n \lambda^{ik}_j =1$.
Note that the  $\lambda^{ik}_j$ are independent of the position of~$f$.  The
parallelotope $R_{ik}$ can now be described in terms of the
$\lambda^{ik}_j$s: \begin{equation}\label{eq:parallelotope}R_{ik} =
  \Bigl\{\, x\in \R^n \mathrel{\Big|} x = f + \sum_{j=1}^n  \mu_j(v^{ij}-f),
  \quad 0 \leq \mu_j \leq \lambda^{ik}_j, \quad j = 1, \ldots, n \,\Bigr\}.\end{equation}

Moreover, expressing the simplicial cone $C_i$ in the form $C_i= \{\,r \in
\R^n \st c^{ij}\cdot r \leq 0, \;\allowbreak j= 1, \ldots, n\,\}$, we can have an inequality description of $R_{ik}$:
\begin{equation}\label{eq:facet-desc}
R_{ik} = \{\,x \in \R^n \st c^{ij}\cdot f \geq c^{ij}\cdot x \geq c^{ij}\cdot y^{ik},\; j = 1, \ldots, n\,\}.
\end{equation}

We define the regions $R_i = \bigcup_{k\in K_i} R_{ik}$ and so $R(f) = \bigcup_{i\in I} R_i$.

\bigskip
We now study $\vol_{\T^n}(R(f)/\Z^n)$ as a function of~$f$.
First of all, we make use of the fact that the volume is additive on sets that
only intersect on the boundary.
\begin{lemma}\label{claim:i1i2}
$\vol_{\T^n}(R(f)/\Z^n) = \sum_{i\in I} \vol_{\T^n}(R_i/\Z^n)$.
\end{lemma}

\begin{proof} To prove the lemma, it suffices to show the following claim.

\bigskip \noindent
{\bf Claim.}
{\em For every $w\in \Z^n$, $\intr(R_{i_1k} + w) \cap \intr(R_{i_2l}) = \emptyset$ whenever $i_1 \neq i_2$ (where $k\in K_{i_1}$ and $l \in K_{i_2}$).}

\bigskip
Suppose to the contrary that there exist two distinct indices $i_1, i_2 \in I$, and $k\in K_{i_1}$, $l\in K_{i_2}$, $x \in \intr(R_{i_1k})$, $w \in \Z^n$ such that  $x + w \in \intr(R_{i_2l})$. For ease of notation and without loss of generality, we assume $i_1 = 1$, $i_2 = 2$.

Using the description of $R_{1k}$ and $R_{2l}$ as given in~\eqref{eq:parallelotope}, there exist multipliers $\mu^1_j$, $j = 1, \ldots, n$ and $\mu^2_j$, $j = 1, \ldots, n$ such that $x = \sum_{j=1}^n \mu^1_j(v^{1j} - f) + f$ with $0 < \mu^1_j < \lambda^{1k}_j$,  $j = 1, \ldots, n$, and $x + w = \sum_{j=1}^n \mu^2_j(v^{2j} - f) + f$ with $0 < \mu^2_j < \lambda^{2l}_j$,  $j= 1, \ldots, n$. Therefore,
\begin{subequations}
  \begin{gather}\label{eq:1}
    \sum_{j=1}^n \mu^1_j(v^{1j} - f) + w = \sum_{j=1}^n \mu^{2}_j(v^{2j} - f)\\
    0 < \mu^{1}_j < \lambda^{1k}_j,\quad 0 < \mu^{2}_j <
    \lambda^{2l}_j\quad\text{for } j = 1,
    \ldots, n
  \end{gather}
\end{subequations}
We observe that by definition of the multipliers $\lambda^{ik}_j$, we have $\sum_{j=1}^n \lambda^{1k}_j = 1$ and $\sum_{j=1}^n \lambda^{2l}_j = 1$. Also, we will assume without loss of generality that $\sum_{j=1}^n \mu^{1}_j \geq \sum_{j=1}^n \mu^{2}_j$, because otherwise we can denote $-w$ by $w$ and the entire proof can be reproduced. We also observe that $\sum_{j=1}^n \lambda^{1k}_j(v^{1j} - f) + f  = y^{1k}\in \Z^n$ by definition. We now show that the lattice point $w + y^{1k} \in \intr(B)$, reaching a contradiction with the fact that $B$ is lattice-free. From (\ref{eq:1}) we deduce that
\begin{align*}
w + y^{1k} & = \sum_{j=1}^n \mu^{2}_j(v^{2j} - f) - \sum_{j=1}^n \mu^{1}_j(v^{1j} - f) + \sum_{j=1}^n \lambda^{1k}_j(v^{1j} - f) + f \\
& = \sum_{j=1}^n \mu^{2}_j(v^{2j} - f) + \sum_{j=1}^n(\lambda^{1k}_j - \mu^{1}_j)(v^{1j} - f) + f
\end{align*}

Using the facts that $\mu^{2}_j > 0$ and $\mu^{1}_j < \lambda^{1k}_j$ for all $j = 1, \ldots, n$, $\sum_{j=1}^n \mu^{1}_j \geq \sum_{j=1}^n \mu^{2}_j$ and $\sum_{j=1}^n \lambda^{1k}_j = 1$, we can rearrange terms and rename the multipliers to obtain
\begin{align*}
 w + y^{ik} & = \sum_{j=0}^n (\delta^1_jv^{1j} + \delta^2_jv^{2j}) + \delta f
\end{align*}
where $\delta^1_j > 0$ and $\delta^2_j > 0$ for all $j$ and $\delta = (1 -  \sum_{j=1}^n \mu^{2}_j - \sum_{j=1}^n \lambda^{1k}_j + \sum_{j=1}^n\mu^{1}_j) \geq 0$ and $\sum_{j=1}^n(\delta^1_j+\delta^2_j) + \delta = 1$. Since the vertices $v^{1j}, v^{2j}, j = 1, \ldots, n$ do not all lie on the same facet, this implies that $w + y^{1k} \in \intr(B)$.
\end{proof}

\bigskip
We now study the volume of each individual region, $\vol_{\T^n} (R_i/\Z^n)$, as a function
of~$f$.  We can measure the volume of $R_i/\Z^n$ in the quotient space by mapping each element $p \in R_i$ to
a canonical representative in the coset $p+\Z^n$, and taking the ordinary
Lebesgue measure of the image.  A suitable canonical map $\phi\colon R_i \to
\R^n$ is defined as follows.
Consider any term order, i.e., a total linear ordering~$\succ$ on~$\R^n$ such
that $w_1 \succ w_2$ implies $w_1 + w \succ w_2 + w$ for all $w_1, w_2, w \in
\R^n$ (e.g., one can use the lexicographical ordering).
Then define $\phi(p)$ as the minimum element in $(p + \Z^n) \cap R_i$ (the
minimum exists since $R_i$ is compact and $p + \Z^n$ is discrete).
Let $\tilde R_i = \phi(R_i) \subseteq \R^n$; then $\vol_{\T^n}(R_i/\Z^n) = \vol(\tilde R_i)$.

We claim that
\begin{equation}\label{eq:modulo}\tilde R_i = R_i \setminus \bigcup_{w \in
    \Z^n : w \succ 0} (R_i + w).
\end{equation}

We first show that $\tilde R_i \subseteq R_i \setminus \bigcup_{w \in \Z^n : w \succ 0} (R_i + w)$. Consider any $p\in R_i$ and let $m =\phi(p)$, i.e., $m$ is the minimum element in $(p + \Z^n) \cap R_i$.  If the
minimum element $m$ is not in $R_i \setminus \bigcup_{w \in \Z^n : w \succ 0} (R_i + w)$ then $m \in R_i + w$ for some $w
\in \Z^n$ with $w \succ 0$ and so $m - w \in R_i$ and $m \succ m - w$ (since
$\succ$ respects addition and $w \succ 0$), contradicting the fact that $m$
was the minimum.

Next, we show that $R_i \setminus \bigcup_{w \in \Z^n : w \succ 0} (R_i + w) \subseteq  \tilde R_i$. Consider any $p \not\in \tilde R_i$. Then either $p \not\in R_i$ or $\phi(p) \neq p$. If $p \not\in R_i$ then certainly $p \not\in R_i \setminus \bigcup_{w \in \Z^n : w \succ 0} (R_i + w)$. If $\phi(p) \neq p$, then there exists $w \succ 0$ such that $p = \phi(p) + w$. Since $\phi(p) \in R_i$ by definition, $p = \phi(p) + w \in R_i + w$ and so $p\not\in R_i \setminus \bigcup_{w \in \Z^n : w \succ 0} (R_i + w)$. This shows that \eqref{eq:modulo} holds.

Since $R_i$ is a subset of the compact set~$B$, there exists a finite
set~$\bar W \subset \Z^n \cap \{ w\succ 0\}$ that is independent of the position of $f$, such that $R_i \cap (R_i + w) =
\emptyset$ if $w\notin \bar W$.  Thus we get the \emph{finite} formula
\begin{equation}\label{eq:modulo2}
  \tilde R_i = R_i \setminus \bigcup_{w \in \bar W} (R_i + w)
  = \Bigl(\bigcup_{k \in K_i} R_{ik}\Bigr)\setminus \Bigl(\bigcup_{k \in K_i}
  \bigcup_{w \in \bar W} (R_{ik} + w)\Bigr),
\end{equation}
which will allow us to use the standard
inclusion--exclusion formula.
We note that the formula is independent of~$f$, but the sets $R_{ik}$
appearing in it depend on~$f$.
For a set $X\subseteq\R^n$, we denote by $[X]$
its indicator function, defined by $[X](x) = 1$ if $x\in X$ and $[X](x) = 0$
otherwise.  Then the inclusion--exclusion formula can be expressed as the
linear identity of indicator functions
\begin{equation}
  \label{eq:incl-excl}
  [\tilde R_i] = \sum_{\gamma \in \bar\Gamma} \mu_\gamma [I_\gamma],
\end{equation}
where $\bar\Gamma$ is some finite index set, each $I_\gamma$ is a polytope
that is the finite
intersection of the form $(R_{ik_1} + t_1) \cap (R_{ik_2} + t_2) \cap \dots
\cap (R_{ik_m} + t_m)$ where $t_1, \ldots, t_m \in \Z^n$, and the
coefficients $\mu_\gamma\in \Z$.  Again the index set~$\bar\Gamma$ and the
coefficients $\mu_\gamma$ are independent of~$f$ because both the index sets
$K_i$ and $\bar W$ are independent of $f$.
The valuation property of the volume (see Chapter I.8 in~\cite{barv})  then implies
\begin{equation}\label{eq:vol-sum}
  \vol_{\T^n}(R_i/\Z^n) = \vol(\tilde R_i) = \sum_{\gamma \in \bar\Gamma}
  \mu_\gamma \vol(I_\gamma).
\end{equation}

We next show that many of the terms $\vol(I_\gamma)$ in~\eqref{eq:vol-sum}
actually vanish because the polytopes $I_\gamma$ are lower-dimensional.  Then we
show that all the remaining terms are affine functions in the coordinates
of~$f$.

To this end, consider the intersection lattice $\Lambda_i = \{\, x\in\Z^n \mid a^i \cdot x =
0 \,\}$, so that $y^{ik} - y^{ik'} \in \Lambda_i$ for any two integer points
$y^{ik}$, $y^{ik'}$ in the facet of $B$ defined by $a^i \cdot (x-f) \leq 1$.

\begin{lemma} \label{lemma:no-self-intersection}
The region $\intr(R_i) \cap \intr(R_i + t) = \emptyset$ for all $t\in \Z^n
\setminus \Lambda_i$.
\end{lemma}
\begin{proof}
Let $t$ be any vector in $\Z^n \setminus \Lambda_i$.
We will show that $\intr(R_i) \cap (R_i + t) = \emptyset$.
Let $D$ be the simplex formed by $f$ and the vertices $v^{ij}$, $j = 1, \ldots, n$; then
$R_i\subseteq D$. We recall that the facet of $B$ determined by $a^i$ is given by $a^i\cdot x \leq b$ where $b = 1 + a^i\cdot f$ and so this is also a facet of $D$.
Note that the sets $R_i+t$ and $R_i$ intersect if and only if their translates $R_i$ and $R_i - t$ intersect.
Therefore we may assume that $t \in \Z^n \setminus \Lambda_i$ has been chosen so that $a^i\cdot t < 0$. Since $R_i \subseteq D$, it suffices to show that $\intr(D) \cap (R_i + t) = \emptyset$.

Consider any parallelotope $R_{ik}$ in $R_i$. We show that $R_{ik} + t$
does not intersect the interior of $D$. Since $D$ is a lattice-free simplex
and $a^i\cdot (t+y^{ik}) < a^i\cdot y^{ik} = b$, there is a facet of $D$
that separates $t + y^{ik}$ from $D$ that is different from the facet given
by $ a^i\cdot x \leq b $. Say this facet is given by $c\cdot x \leq d$. Therefore, $d \leq c\cdot (t+y^{ik})$. Using \eqref{eq:facet-desc}, there is a facet of $R_{ik}$ given by $c\cdot x \geq c\cdot y^{ik}$ and so $R_{ik} + t$ has a facet given by $c\cdot x \geq c\cdot (t+y^{ik})$. Since $D$ has a facet $c\cdot x \leq d$ and $d \leq c\cdot (t+y^{ik})$, $\intr(D) \cap (R_i + t) = \emptyset$.
\end{proof}

Now let us define $\Gamma$ as the subset of $\gamma\in\bar\Gamma$ for which
$I_\gamma$ is a finite intersection of the form $(R_{ik_1} + t_1) \cap
(R_{ik_2} + t_2) \cap \dots \cap (R_{ik_m} + t_m)$ where the translation
vectors $t_1, \ldots, t_m$ all lie in the intersection lattice~$\Lambda_i$.
By Lemma~\ref{lemma:no-self-intersection}, if $\gamma\in
\bar\Gamma\setminus\Gamma$, then $\vol(I_\gamma) = 0$.
Thus we obtain the formula
\begin{equation}\label{eq:vol-sum-reduced}
  \vol_{\T^n}(R_i/\Z^n) = \vol(\tilde R_i) = \sum_{\gamma \in \Gamma}
  \mu_\gamma \vol(I_\gamma).
\end{equation}

We now consider the remaining terms~$\vol(I_\gamma)$ for $\gamma\in \Gamma$.
Since $\Lambda_i$ is the intersection lattice of the linear space spanned by $v^{ij} - v^{i1}$, $j = 2, \ldots, n$, the coordinates of any $t \in \Lambda_i$ in the basis $v^{ij} - f$, $j= 1, \ldots, n$, are independent of $f$. Therefore, for arbitrary $t_1, t_2 \in \Lambda_i$ and $R_{ik_1}$ and $R_{ik_2}$ we have that $(R_{ik_1} + t_1) \cap  (R_{ik_2} + t_2)$ is of the form
\begin{equation}
 \Bigl\{\, x\in \R^n \mathrel{\Big|} x = f + \sum_{j=1}^n
 \mu_j(v^{ij}-f), \quad \mu^{ih}_j \leq \mu_j \leq \nu^{ih}_j, \quad j = 1,
 \ldots, n \,\Bigr\}, \label{eq:elem-paras}
\end{equation}
where $\mu^{ih}_j,
\nu^{ih}_j$ are independent of~$f$. By induction on the number of terms appearing in the intersection expression for $I_\gamma$, we deduce that $I_\gamma$ is also a parallelotope of the form  \eqref{eq:elem-paras}. We now show that the volume of $I_\gamma$ is an affine function of the
coordinates of~$f$. To this end, for $i\in I$, let $M_i$ be the $(n+1) \times
(n+1)$ matrix with columns $v^{ij}$, $j=1, \ldots, n$, and
$f$ expressed in homogeneous coordinates (therefore all the entries in the
last row of $M_i$ are 1). Since $I_\gamma$ is a parallelotope of the form \eqref{eq:elem-paras}, we have that $\vol(I_\gamma) = \lvert \det(M_i)\rvert \prod_{j=1}^n (\nu^{ik}_j-\mu^{ik}_j)$.  Since $\mu^{ik}_j,
\nu^{ik}_j$ are independent of~$f$, this implies the following.

\begin{lemma}\label{lem:linear-vol}
$\vol(I_\gamma)$ is an affine function of the coordinates of $f \in \intr(B)$ for all $\gamma \in \Gamma$.
\end{lemma}

Combining Lemma~\ref{claim:i1i2} with equation~\eqref{eq:vol-sum} and Lemma~\ref{lem:linear-vol}, we have Theorem~\ref{lem:mod_vol}.

\bigskip
So far, we have assumed that $f$ is in the interior of $B$. Since $\vol_{\T^n}(R(f)/\Z^n)$ is an affine function of the coordinates of $f$ by Theorem~\ref{lem:mod_vol}, it is natural to extend its definition to points $f$ on the boundary of $B$ by continuity. In fact, when $f$ is on the boundary of~$B$, the region~$R$ can still be defined using \eqref{eq:R}
 and \eqref{eq:parallelotope}. The regions $R_{ik}$ that have positive volume correspond to integer points $y^{ik}$ in the interior of the facets of $B$ that do not contain $f$. Theorem~\ref{lem:mod_vol} implies the following.

\begin{cor}\label{cor:face}
Let $B$ be a maximal lattice-free simplicial polytope in $\R^n$.
Then the set $\{\, f\in B\st \vol_{\T^n}(R(f)/\Z^n) = 1 \,\}$ is a face of $B$.
\end{cor}

\begin{proof}
Since
$\vol_{\T^n}(R(f)/\Z^n)$ is always at most 1, the value 1 is a maximum value for the function
$\vol_{\T^n}(R(f)/\Z^n)$. By Theorem~\ref{lem:mod_vol}, optimizing this function over $B$ is a linear
 program and hence the optimal set is a face of $B$.
\end{proof}

We deduce that it suffices to check covering properties at the vertices when we try to decide if the minimal inequality derived from a simplicial polytope has a unique minimal lifting.

Corollary~\ref{cor:face} implies Theorem~\ref{body-with}. Indeed, if the set $\{\, f\in B\st \vol_{\T^n}(R(f)/\Z^n) = 1 \,\}$ is~$B$,
then $R(f)+\Z^n=\R^n$ for all $f \in B$, and therefore $B$ is a body with a
unique lifting for all $f \in \intr(B)$. Otherwise, $R(f)+\Z^n\not=\R^n$ for
all $f \in \intr(B)$, and therefore $B$ is a body with multiple liftings for all $f \in \intr(B)$.

\section{Simplices with exactly one lattice point in the relative interior of each facet}\label{sec:one-latt}

We prove Theorem~\ref{thm:main} in this section.
 Let $\Delta$ be a simplex in $\R^n$ ($n \geq 2$) such that it is a maximal lattice-free convex body and each facet of $\Delta$ has exactly one integer point in its relative interior. Note that there is no restriction on the number of integer points on faces of dimension $n-2$ or lower. We consider the region $R(f)$ as defined in \eqref{eq:R} for the minimal valid function derived from this simplex and we show that the volume of this region is at most~1.

By making an affine unimodular transformation, we can assume that one of the lattice points in the relative interior of a facet of $\Delta$ is $0$. Let $y^1, \ldots, y^n$ be the other integer points in the relative interior of the facets of $\Delta$. We extend the notation and denote $y^0 = 0$. \old{We label the remaining integer point on the boundary of $\Delta$ by $y^{n+1} , \ldots, y^k$.}With a slight extension of our earlier notation, let us denote by $a^i, i = 0,\ldots, n$, the normals to the facet of $\Delta$ passing through $y^i$ such that $$\Delta = \{\,x \in \R^n \st a^i\cdot x \leq a^i\cdot y^i \;\;\;\; \forall i = 0, \ldots, n\,\}.$$
Let $v^0, \ldots, v^n$ be the vertices of $\Delta$, labeled such that the facet not containing $v^i$ contains~$y^i$. As per our earlier notation, for $i = 0, \ldots, n$,  $R_i$ will denote the part of $R(f)$ formed by the integer points lying on the facet $a^i \cdot x \leq a^i \cdot y^i$. Also, observe that the only parallelotope in $R_i$ with non-zero volume is the one with $y^i$ as a vertex. Henceforth, we will refer to this parallelotope with nonzero volume as $R_i$ itself, since the other regions do not contribute anything.

To prove Theorem~\ref{thm:main}, it suffices to prove the following result.

\begin{theorem}\label{thm:vol}
Let $f= v^0$. Then $\vol(R_0) \leq 1$, where equality holds if and only if $\Delta$ is an affine unimodular transformation of $\conv\{0, ne^1, \ldots, ne^n\}$.
\end{theorem}

Using this, we first prove Theorem~\ref{thm:main}.

\begin{proof}[Proof of Theorem~\ref{thm:main}]
When $\Delta$ is an affine unimodular transformation of
the simplex $\conv\{0, ne^1, \ldots, ne^n\}$, consider $f$ at one of the vertices, say $v^0$, which is the vertex corresponding to $0$. Then $R_0$ is simply an affine unimodular transformation of the cube, and therefore $\vol_{\T^n} (R_0/\Z^n) = 1$. The same holds for $f$ at the other vertices of $\Delta$, which implies by
Theorem~\ref{lem:mod_vol} that $\vol_{\T^n} (R(f)/\Z^n) = 1$ for all $f \in \Delta$, i.e., $\Delta$ is a body with a unique lifting for all $f \in \intr(\Delta)$.
 We note here that this fact was also proved in~\cite{ccz}. Theorem~\ref{lem:mod_vol} enables us to give a more immediate proof.

Now suppose that $\Delta$ is a body with a unique lifting for all $f \in \intr(\Delta)$. Then when $f = v^0$ we have $\vol_{\T^n} (R_0/\Z^n) = 1$, and therefore $\vol (R_0) \geq 1$. By the first part of Theorem~\ref{thm:vol} we have $\vol (R_0) = 1$. Then by the second part of Theorem~\ref{thm:vol}, we know that $\Delta$ is an affine unimodular transformation of $\conv\{0, ne^1, \ldots, ne^n\}$.
\end{proof}

We state two classical theorems from the geometry of numbers that we will need to prove Theorem~\ref{thm:vol}.  Let $S$ be a convex body symmetric about the origin such that $\intr(S) \cap \Z^n = \{0\}$. Such bodies are called $0$-symmetric lattice-free convex bodies.

\begin{theorem}[Minkowski's Convex Body Theorem, see \cite{gruber}, Theorem 1, Section 5, Chapter 2]\label{thm:minkowski1}
 If $S$ is a $0$-symmetric lattice-free convex body, then $\vol(S) \leq 2^n$.
\end{theorem}

If $\vol(S) = 2^n$ for some $0$-symmetric lattice-free convex body $S$, then $S$ is called an \emph{extremal} $0$-symmetric lattice-free convex body (see \cite{gruber}, Section 12, Chapter 2).

\begin{theorem}[Minkowski--Haj\'os, see \cite{gruber}, Section 12.4, Chapter 2]\label{thm:minkowski3}
Let $S$ be an extremal $0$-symmetric lattice-free parallelotope. Then there exists a unimodular transformation $U$ such that after applying $U$, $S$ will have two parallel facets given by $-1 \leq x_1 \leq 1$.
\end{theorem}

We use this theorem to prove a lemma about extremal $0$-symmetric lattice-free parallelotopes.

\begin{lemma}\label{lem:extremal_parallel}
Let $S \subseteq \R^n$ be an extremal $0$-symmetric lattice-free parallelotope. If every facet of $S$ has at most (and therefore exactly) one integral point in its relative interior, then $S$ is a unimodular transformation of the cube $\{-1 \leq x_i \leq 1,\; i = 1, \ldots, n\}$.
\end{lemma}

\begin{proof}
We prove this by induction on the dimension $n$. For $n=1$, this is trivial. Consider $n \geq 2$. The Minkowski--Haj\'os theorem implies that we can make a unimodular transformation such that $S$ = $\conv\{(S\cap \{x_1 = -1\}) \cup (S\cap \{x_1 = 1\})\}$. Note that $S\cap \{x_1 = -1\}$, $S\cap \{x_1 = 0\}$ and $S\cap \{x_1 = 1\}$ are all translations of each other. Therefore, $2^n = \vol(S)= 2\vol(S\cap \{x_1 = 0\})$ (here we measure volume of $S\cap \{x_1 = 0\}$ in the $n-1$ dimensional linear space $x_1 = 0$). Therefore, $S\cap \{x_1 = 0\}$ has volume $2^{n-1}$. Therefore, $S\cap \{x_1 = 0\}$ is also an extremal $0$-symmetric lattice-free parallelotope in the linear space $x_1=0$. If any facet $F$ of  $S\cap \{x_1 = 0\}$ contains 2 or more lattice points in its relative interior, then the facet of $S$ passing through $F$ will contain these lattice points in its relative interior, in contradiction to the hypothesis of the theorem. Therefore, every facet of $S\cap \{x_1 = 0\}$ contains at most one lattice point in its interior. By the induction hypothesis, $S\cap \{x_1 = 0\}$ is equivalent to the cube $\{-1 \leq x_i \leq 1,\; i= 2, \ldots, n\} \cap \{x_1 = 0\}$. Finally, observe that $S\cap \{x_1 = -1\}$ and $S\cap \{x_1 = 1\}$ are translations of $S\cap \{x_1 = 0\}$. Since $S\cap \{x_1 = 0\}$  is equivalent to the cube $\{-1 \leq x_i \leq 1,\; i= 2, \ldots, n\} \cap \{x_1 = 0\}$, any translation by a non lattice vector will contain at least 2 integral points in its relative interior. But the facets $S\cap \{x_1 = -1\}$ and $S\cap \{x_1 = 1\}$ contain at most 1 lattice point in their relative interior. Therefore, they are in fact lattice translates of $S\cap \{x_1 = 0\}$. This proves the lemma.
\end{proof}

\bigskip
\begin{proof}[Proof of Theorem~\ref{thm:vol}]

Since $f = v^0$, using \eqref{eq:facet-desc} the parallelotope $R_0$ can also be described as
\[\begin{array}{rcl} R_0 & = & \{\,x \in \R^n \st 0 \leq a^i\cdot x \leq a^i\cdot v^0 \;\;\;\; \forall i = 1, \ldots, n\,\} \\ & = &\{\,x \in \R^n \st 0 \leq a^i\cdot x \leq a^i\cdot y^i \;\;\;\; \forall i = 1, \ldots, n\,\},\end{array}\]

since $a^iv^0 = a^iy^i$. We consider the set $S = \{x \in \R^n \st -a^i\cdot y^i \leq a^i\cdot x \leq a^i\cdot y^i\}$. Clearly $S$ is a centrally symmetric convex body containing the origin. Moreover, $\vol(S) = 2^n\vol(R_0)$. We first claim

\begin{claim}
$\intr(S)\cap \Z^n = \{0\}. $
\end{claim}

\begin{proof}
Suppose to the contrary that $p \in \intr(S)\cap \Z^n$ and $p\neq0$. Since $S$ is centrally symmetric, $-p  \in \intr(S)\cap \Z^n$. Note that this implies that both $p$ and $-p$ satisfy $ a^i\cdot x < a^i\cdot y^i$ for all $i=1, \ldots, n$. Moreover, for at least one of $p$ and $-p$, the inequality $a^0\cdot x \leq 0$ is valid. This shows that either $p$ or $-p$ is in the interior of $\Delta$ or both of them lie on the hyperplane $a^0\cdot x = 0$ and so both of them lie in the relative interior of the facet corresponding to $a^0\cdot x \leq 0$. This is a contradiction.
\end{proof}

By this claim, Theorem~\ref{thm:minkowski1} applies to $S$ and we conclude that $\vol(S) \leq 2^n$. Since $\vol(S) = 2^n\vol(R_0)$, we immediately conclude that $\vol(R_0) \leq 1$.

Now we prove the second part of the theorem. If $\Delta$ is an affine
unimodular transformation of $\conv\{0, ne^1, \ldots, ne^n\}$, then up to an
affine unimodular transformation the region $R_0$ is given by the cube $\{\,x \st 0 \leq x_i \leq 1, \;i = 1, \ldots, n\,\}$, and therefore $\vol(R_0) = 1$. For the converse, let us assume, $\vol(R_0) = 1$ and so $\vol(S) = 2^n$. So $S$ is an extremal parallelotope. Now we claim that $S$ has at most one integral point in the relative interior of each facet. Suppose not and let $p_1$ and $p_2$ be two points in the relative interior of a facet. By the centrally symmetric property, we can assume that in fact {\em two} parallel facets of $S$, $a^i \cdot x \leq a^i \cdot y^i$ and $-a^i \cdot y^i \leq a^i \cdot x$ have two integral points in their relative interior. We label these integral points as $p_1$ and $p_2$ lying on $a^i \cdot x \leq a^i \cdot y^i$, and so $-p_1$ and $-p_2$ lie on $-a^i \cdot y^i \leq a^i \cdot x$. Now the inequality $a^i \cdot x \leq a^i \cdot y^i$ defines a facet of $\Delta$, and since this facet of $\Delta$ has exactly one integral point in its relative interior, either $p_1$ or $p_2$ needs to satisfy $a^0\cdot x \geq 0$. Suppose $p_1$ satisfies this inequality, then $a^0\cdot (-p_1) \leq 0$ and therefore $-p_1$ either lies in the interior of $\Delta$ or in the relative interior of the facet $a^0 \cdot x \leq 0$. In either case, we have a contradiction to the assumptions on $\Delta$.

So we conclude that $S$ has at most one integral point in the relative
interior of each facet. By Lemma~\ref{lem:extremal_parallel}, we have that $S$
is equivalent to the cube $\{\,-1 \leq x_i \leq 1, \;i = 1, \ldots, n\,\}$ up to
unimodular transformation. Therefore, up to a unimodular transformation, $\Delta$ has $n$~facets given by $\{\,x_i \leq 1, \;i = 1, \ldots, n\,\}$. Now we finally prove that the facet $F$ of $\Delta$ passing through $0$ is given by $x_1 + x_2 + \ldots + x_n \geq 0$. This will complete the proof of the second part of Theorem~\ref{thm:vol}.

We first claim that the standard unit vector $e^i$ is in the relative interior of the facet $x_i \leq 1$ of $\Delta$. This is because the point $-e^i$ is strictly cut off by the facet $F$ of $\Delta$ passing through $0$, since $\Delta$ has no integer point in its interior and $0$ is the unique integer point in the relative interior of $F$. So $e^i$ has to be in the relative interior of the facet $x_i \leq 1$ of~$\Delta$.

Consider the following set of linearly independent points $p^i, i=1, \ldots, n-1$ given by

\[
p^i_j = \left\{ \begin{array}{rl} 1 & j = i \\ -1 & j = n \\ 0 & j \neq i, n.  \end{array}\right.
\]

Unless the facet $F$ of $\Delta$ passing through $0$ is given by $x_1 + x_2 + \ldots + x_n \geq 0$, at least one of $p^i$ or $-p^i$ is strictly valid for~$F$ (since this is a set of $2(n-1)$ points  with affine dimension $n-1$, all lying on $x_1 + x_2 + \ldots + x_n = 0$), say for $i=k$. If $p^k$ is strictly valid for~$F$, then $p^k$ lies in the relative interior of the facet $x_k \leq 1$ of $\Delta$. If $-p^k$ is strictly valid for~$F$, then $-p^k$ lies in the relative interior of the facet $x_n \leq 1$ of $\Delta$. This would be a contradiction to the assumption that each facet of $\Delta$ contains exactly one integer point in its relative interior, since we observed earlier that $e^i$ is certainly in the relative interior of the facet $x_i \leq 1$.
\end{proof}
\bigskip

\section{2-Partitionable Simplices} \label{sec:multi}

\subsection{Proof of Theorem~\ref{thm:multiple}}

Let $\Delta \subset \R^{n+1}$ be a maximal lattice-free 2-partitionable simplex.
Without loss of generality we may assume that the consecutive lattice hyperplanes in the 2-partition are $x_{n+1} = 0$ and $x_{n+1} = -1$. We further assume that $x_{n+1} \geq -1$ defines a facet of $\Delta$.
 Let $F_{n+1} = \Delta \cap \{x_{n+1} = -1\}$ denote this facet.

Suppose that $\Delta \cap \{x_{n+1} = 0\}$ is {\em not} an affine unimodular
transformation of the simplex $\conv\{0, ne^1, \ldots, ne^n\}$. Then consider the lifting
region $R(\hat v)$ when $f$ is one of the vertices~$\hat v$ of the facet~$F_{n+1}$.
Consider the vertex $\bar v$ of  $\Delta \cap \{x_{n+1} =
0\}$ lying on an edge incident to $\hat v$. Since $\Delta$ is 2-partitionable, $\Delta \cap \{x_{n+1} =
0\}$ is a maximal lattice-free simplex when viewed as embedded in $\R^n$. Indeed, it is lattice-free because $\Delta$ is lattice-free and each facet of $\Delta \cap \{x_{n+1} = 0\}$ contains a lattice point in its relative interior, and hence it is a maximal lattice-free set. Considering $\Delta \cap \{x_{n+1} =
0\}$ as a maximal lattice-free simplex in $n$ dimensions, let $\bar R$ be the region \eqref{eq:R} formed by the vertex $\bar v$. Note that the region $R(\hat v)$ is a
cylinder over $\bar R$ with height 1. Theorem~\ref{thm:main} implies that
$\bar R$ has $n$-dimensional volume modulo lattice translations strictly less
than 1 and therefore, so does $R(\hat v)$. So for every vertex~$v$ except
possibly
the one not lying on the facet $F_{n+1}$, we have $\vol_{\T^{n+1}}(R(v)/\Z^{n+1}) < 1$. Combined with Theorem~\ref{lem:mod_vol}, this implies that $\vol_{\T^{n+1}}(R(f)/\Z^{n+1}) < 1$ for all $f \in \intr(\Delta)$. So $\Delta$ is a body with multiple liftings.

Suppose now that $\Delta \cap \{x_{n+1} = 0\}$ is an affine unimodular
transformation of the simplex $\conv\{0, ne^1, \ldots, ne^n\}$. We will show
that the region $R(f)$ is such that we have $\vol_{\T^{n+1}}(R(f)/\Z^{n+1}) =
1$ when $f$ is at the vertices. As observed in the first case, for the
vertices on the facet $F_{n+1}$, the set $R(f)$ is simply a cylinder of
height~1 over a region $\bar R$ from the simplex $\Delta \cap \{x_{n+1} =
0\}$ defined by one of its vertices. Again by Theorem~\ref{thm:main} we know that this has volume 1 modulo
the lattice.

We finally argue about the vertex $v$ that does not lie on the facet $F_{n+1}$
and show that $\vol_{\T^{n+1}}(R(v)/\Z^{n+1}) = 1$. By
Theorem~\ref{lem:mod_vol}, this will imply that
$\vol_{\T^{n+1}}(R(f)/\Z^{n+1}) = 1$ for all $f \in \intr(\Delta)$. So
$\Delta$ is a body with a unique lifting for all $f \in \intr(\Delta)$.

Let $S_0 = \Delta \cap \{x_{n+1} = 0\}$. By a unimodular transformation, we
assume that $S_0=\conv\{0, ne^1, \ldots, ne^n\}$. This then implies that the
facet description of $\Delta$ is given by the following set of inequalities:
\begin{subequations}
  \begin{alignat}{2}
    -x_i + \delta_ix_{n+1} & \leq 0 &&\qquad \text{for } i \in \{1, \ldots, n\}, \label{eq:facet-D1}\\
    \sum_{i=1}^nx_i + \delta_{n+1}x_{n+1} & \leq n, \label{eq:facet-D2}\\
    x_{n+1} &\geq -1,\label{eq:facet-D3}
  \end{alignat}
\end{subequations}
where $\delta_i$, $i = 1, \ldots, n+1$, are some constants in $\R$.

\begin{claim}
The facet $F_{n+1}$ contains a copy of $S_0$ translated by an integer vector.
\end{claim}

\begin{proof}
Consider the point $p^1 = (- \lfloor \delta_1 \rfloor, - \lfloor \delta_2 \rfloor,
\ldots, - \lfloor \delta_n \rfloor, -1)$ and $p^2 = (\lfloor \delta_1 \rfloor+1,
\lfloor \delta_2 \rfloor+1, \ldots, \lfloor \delta_n \rfloor+1, +1)$. We first
observe that $p^2$ satisfies the constraints \eqref{eq:facet-D1}
strictly. Indeed, $-p^2_i + \delta_ip^2_{n+1} = - \lfloor \delta_i \rfloor - 1 +
\delta_i < 0$. Since $\Delta$ is lattice-free, and $p^2 \in \Z^{n+1}$ satisfies
\eqref{eq:facet-D1} and \eqref{eq:facet-D3} strictly, it cannot satisfy
inequality \eqref{eq:facet-D2} strictly. This implies that
\begin{equation}\label{eq:delta}
\delta_{n+1} \geq - \sum_{i=1}^n\lfloor\delta_i\rfloor.
\end{equation}

Now we claim that $p^1$ and $p^1 + ne^j$ for  $j=1, \ldots, n$ all satisfy the
constraints \eqref{eq:facet-D1} and~\eqref{eq:facet-D2}, and satisfy \eqref{eq:facet-D3} at equality,
and hence all lie on the facet $F_{n+1}$. This will prove the
claim. Clearly, \eqref{eq:facet-D3} is satisfied at equality for all of these
$n+1$ points. We now check \eqref{eq:facet-D1} and
\eqref{eq:facet-D2}. Consider~$p^1$. Since $-p^1_i  + \delta_ip^1_{n+1} =
\lfloor \delta_i \rfloor - \delta_i \leq 0$, the point~$p^1$ satisfies
\eqref{eq:facet-D1}. Next, $\sum_{i=1}^np^1_i + \delta_{n+1}p^1_{n+1} =
- \sum_{i=1}^n\lfloor \delta_i \rfloor - \delta_{n+1} \leq 0 \leq n$ by
\eqref{eq:delta}, showing that \eqref{eq:facet-D2} is satisfied. Consider $p^1
+ ne^j$. Evaluating the left-hand side of \eqref{eq:facet-D1}, we get no
change from the computation for $p^1$ except when $i=j$. In this case, we find
$\lfloor \delta_i \rfloor - \delta_i - n \leq 0$. Finally for constraint
\eqref{eq:facet-D2}, we have $ - \sum_{i=1}^n\lfloor \delta_i \rfloor -
\delta_{n+1} + n \leq n$ because of \eqref{eq:delta}.
\end{proof}

The above claim implies that after applying a unimodular transformation~$U$,
the translated set $S_0 - e^{n+1}$  is the copy of $S_0$ lying on the facet
$F_{n+1}$ (i.e., the two copies lie vertically on top of each other). This in
turn implies that after applying~$U$, the simplex~$\Delta$ can be described by \eqref{eq:facet-D1}, \eqref{eq:facet-D2} and \eqref{eq:facet-D3} with $\delta_i \geq 0$ for all $i=1,\ldots, n+1$, since the points $(0, 0,\ldots, 0, -1), (n, 0,\ldots, 0, -1) \in S_0 - e^{n+1}$ satisfy all the constraints \eqref{eq:facet-D1},  \eqref{eq:facet-D2} and~\eqref{eq:facet-D3}.

Along with the standard unit vectors $e^1, \ldots, e^{n+1}$, we use $e^0$ to
denote the origin $0$. Let $d = \sum_{i=1}^{n+1}\delta_i$. Then $d>0$ because
otherwise the system~\eqref{eq:facet-D1}, \eqref{eq:facet-D2} and
\eqref{eq:facet-D3} defines a cylinder over $S_0$ and not a simplex.
Let $\hat f \in \R^{n+1}$ be the point given by $\hat f_{n+1} = 0$ and $\hat
f_{i} = \frac{\delta_i n}{d}$. Since $\hat f_i \geq 0$ for all $i \in \{1, \ldots, n\}$ and $\sum_{i=1}^n \hat f_i\leq n$, we have
$\hat f \in S_0$. We define the following $n+1$ regions: $$D_k = \Bigl\{\,x \in \R^{n+1}
\mathrel{\Big|} x = \hat f + \sum_{j\in \{0, \ldots, n\}\setminus \{k\}}\!\!\!\!\! \alpha_j\bigl(e^j -
\tfrac1{n}{\hat f}\bigr), \; 0 \leq \alpha_j \leq 1\,\Bigr\} \quad\text{for } k=0,\ldots, n.$$

Note that each $D_k \subseteq \R^{n+1}\cap \{x_{n+1} = 0\}$. We consider $S_0$ as an $n$-dimensional simplex embedded in $\R^{n+1}$ containing $\hat f$, then we can define the region $R(\hat f)$ for $S_0$ as in \eqref{eq:R}. If we view $R(\hat f)$ as a set embedded in $\R^{n+1}$, then we observe that each $D_k$ is a parallelotope that appears in the union defining $R(\hat f)$. In fact, all other parallelotopes in the definition of $R(\hat f)$ are of dimension $n-2$. Therefore, if we view $R(\hat f)$ and $\bigcup_k D_k$ as sets of $\R^n$, $\vol_{\T^n}((\bigcup_k D_k)/\Z^n) = \vol_{\T^n}(R(\hat f))/\Z^n) = 1$ by Theorem~\ref{thm:main}. We now show that for every point $x$ in $\bigcup_k D_k$, there exists a segment $\{\, x + \lambda e^{n+1} \in \R^{n+1} \st \lambda_0 \leq \lambda \leq \lambda_0+1 \,\}$ of height~1 above $x$ that lies in $R(v)$ (recall $v$ is the vertex of $\Delta$ not on the facet $F_{n+1}$). This will then show that $\vol_{\T^{n+1}}(R(v)/\Z^{n+1})=1$ and we will be done.

Consider any $x \in D_k$ for some $k \in \{1, \ldots, n\}$. This means
\begin{displaymath}
  x = \hat f + \sum_{j\in \{0, \ldots, n\}\setminus \{k\}}\!\!\!\!\!\alpha_j\bigl(e^j - \tfrac1{n}{\hat f}\bigr), \quad 0 \leq \alpha_j \leq 1.
\end{displaymath}
Let $\alpha = \sum_{j\in \{0, \ldots, n\}\setminus \{k\}}\alpha_j$. Then $x = (1-\frac{\alpha}{n})\hat f + \sum_{j\in \{0, \ldots, n\}\setminus \{k\}}\alpha_j e^j$. Consider the two points $x^1, x^2 \in \R^{n+1}$ given by $x^1 = x +  \bigl((1-\frac{\alpha}{n})\frac{n}{d} - 1\bigr)e^{n+1}$ and $x^2 = x + \bigl((1-\frac{\alpha}{n})\frac{n}{d}\bigr)e^{n+1}$. Consider the parallelotope $\hat R$ given by all points $x\in \R^n$ such that
\begin{subequations}
  \begin{alignat}{2}
    -1 - \delta_i \leq -x_i + \delta_i x_{n+1} &\leq 0 &\quad &\text{for } i \in \{1,\ldots, n\}\setminus \{k\}, \label{eq:R1} \\
    -\delta_k \leq -x_k + \delta_kx_{n+1} &\leq 0, & & \label{eq:R2}\\
    n-1 - \delta_{n+1} \leq \sum_{i=1}^nx_i + \delta_{n+1}x_{n+1} &\leq n. &
    & \label{eq:R3}
  \end{alignat}
\end{subequations}

The parallelotope $\hat R$ is one of the form \eqref{eq:facet-desc} with
$y^{ik} = e - e^k - e^{n+1}$, where $e$ is the vector of all 1's except the
last coordinate, which is 0. This is a lattice point in $S_0 - e^{n+1}$. Therefore, $\hat R \subseteq R(v)$. We
now check that $x^1$ and $x^2$ both satisfy the constraints \eqref{eq:R1},
\eqref{eq:R2} and \eqref{eq:R3}. We do the computation for $x^1$; the case of $x^2$ can be argued similarly. Plugging $x^1$ into the middle term in \eqref{eq:R1}, we get $$- \bigl(1-\tfrac{\alpha}{n}\bigr)\tfrac{\delta_i n}{d} - \alpha_i + \delta_i\bigl( (1-\tfrac{\alpha}{n})\tfrac{n}{d} - 1)\bigr) = -\alpha_i - \delta_i.$$ Since $0 \leq \alpha_i \leq 1$ and $\delta_i \geq 0$, the constraints \eqref{eq:R1} are satisfied.  Plugging $x^1$ into the middle term in \eqref{eq:R2}, we get $$- \bigl(1-\tfrac{\alpha}{n}\bigr)\tfrac{\delta_k n}{d} + \delta_k\bigl( (1-\tfrac{\alpha}{n})\tfrac{n}{d} - 1)\bigr) = -\delta_k$$ and therefore \eqref{eq:R2} is satisfied. Finally, plugging $x^1$ into the middle term in \eqref{eq:R3}, we get $$\textstyle\bigl(1-\frac{\alpha}{n})(\sum_{i=1}^n \delta_i)\tfrac{n}{d} + \sum_{j\in\{1, \ldots, n\}\setminus\{k\}}\alpha_j +  \delta_{n+1}\bigl( (1-\frac{\alpha}{n})\frac{n}{d} - 1)\bigr) = n - \alpha_0 - \delta_{n+1}.$$ Using $0 \leq \alpha_0 \leq 1$ and $\delta_{n+1} \geq 0$, we see that \eqref{eq:R3} is satisfied. Since $\hat R$ is a convex parallelotope, the entire segment $[x^1, x^2]$ of height~1 is contained in $\hat R$ and hence in $R(v)$ because $\hat R \subseteq R(v)$.

We finally deal with $x \in D_0$. For this purpose we use the parallelotope
\begin{alignat}{2}
-1 - \delta_i \leq -x_i + \delta_i x_{n+1} &\leq 0 & \quad& \text{for } i \in \{1,\ldots, n\},\\
n - \delta_{n+1} \leq \sum_{i=1}^nx_i + \delta_{n+1}x_{n+1}  &\leq  n, &&
\end{alignat}
which is again obtained using the form \eqref{eq:facet-desc} with $y^{ik} = e
- e^{n+1}$. Now the point $x$ can be represented as $x = \hat f\bigl(1 - \frac1n \sum_{j=1}^n\alpha_j\bigr) + \sum_{j = 1}^n\alpha_j e^j$. We set $\alpha = \sum_{j=1}^n\alpha_j$ and $x^1 = x + (\frac{n}{d} - 1 -\frac{\alpha}{d})e^{n+1}$ and $x^2 =  x + (\frac{n}{d} -\frac{\alpha}{d})e^{n+1}$. Similar computations as before can be carried out to check that the entire segment $[x^1,x^2]$ of height~1 lies in this parallelotope.

\subsection{A simplex with multiple integer points in the interior of a facet with lifting volume less than 1}\label{sec:construction}

The proof of Theorem~\ref{thm:multiple} also shows how to explicitly construct a simplex~$\Delta\subseteq\R^{n+1}$ with more than
one lattice point in the relative interior of a facet and $\vol(R(f))<1$ when
$f$ is one of the vertices. Below we make the construction of such an example
more concrete for the case $n=2$. Observe that for $n=1$, we know that any maximal lattice-free triangle $\Delta \subset \R^2$ with more than one lattice point in the relative interior of a side is a Type 2 triangle, according to the classification by Dey and Wolsey in~\cite{dw}. Then $\Delta \cap H_2$ as described in the statement of Theorem~\ref{thm:multiple} is always a unimodular transformation of $\conv\{0, e^1\}$. So for $n=1$, we cannot construct such an example with multiple liftings.

We construct $\Delta \subset \R^3$ as follows.  One facet is given by $x_1 \geq 0$. $\Delta \cap \{x_1 = 1\}$ is a Type~3 triangle (following the classification of Dey and Wolsey~\cite{dw}). $\Delta \cap \{x_1 = 0\}$ is a homothetic copy of this Type~3 triangle, blown up by a very large factor (hence, the number of integral points in the interior of this facet can be made arbitrarily large). This implies that the vertex of $\Delta$ opposite the facet $x_1 \geq 0$ is very close to the plane $x_1 = 1$. Now when $f$ is one of the vertices of $\Delta$ on the facet $x_1 \geq 0$, as observed in the proof of Theorem~\ref{thm:multiple}, the lifting region is a cylinder of height~1 over the lifting region of the Type~3 triangle $\Delta \cap \{x_1 = 1\}$ which has area $< 1$, and so the volume of the lifting region for $\Delta$ with this position of $f$ is less than 1.

This construction can be generalized to higher dimensions where $\Delta \cap \{x_1 = 1\}$ is a maximal lattice-free $(n-1)$-simplex not equal to an affine unimodular transformation of $\conv\{0, (n-1)e^1, \ldots, (n-1)e^{n-1}\}$.

\section*{Acknowledgments} We would like to thank Christian Wagner and two anonymous referees for many insightful and helpful comments, which have helped us to make the presentation of our results more precise and clear.

\end{document}